\documentclass[11pt, a4paper,reqno]{amsart}
\usepackage{a4,amsmath,amssymb,eepic}
\usepackage{graphicx}
\usepackage{multicol}
\usepackage{rotating}
\usepackage{lscape}
\usepackage{amssymb,amsthm}
\usepackage{enumitem}
\usepackage[all]{xy}
\usepackage{tikz}
\usepackage{color}

\allowdisplaybreaks
\numberwithin{equation}{subsection}
\theoremstyle{plain}
\newtheorem{lem}{Lemma}[subsection]
\newtheorem{prop}[lem]{Proposition}

\newtheorem{thm}[lem]{Theorem}
\newtheorem{cor}[lem]{Corollary}

\newtheorem{ass}[lem]{Assumption}

\theoremstyle{remark}

\newtheorem{defn}[lem]{Definition}

\newtheorem{eg}[lem]{Example}
\newtheorem{rem}[lem]{Remark}

\newcommand{\Ext}{\operatorname{Ext}}
\newcommand{\Hom}{\operatorname{Hom}}
\newcommand{\ind}{\operatorname{ind}}

\newcommand{\res}{\operatorname{res}}
\newcommand{\ZZ}{{\mathbb Z}}
\newcommand{\NN}{{\mathbb N}}
\newcommand{\remb}{\operatorname{rem}}
\newcommand{\add}{\operatorname{add}}

\begin{document}

\title[Cyclotomic Brauer algebras]{Decomposition numbers for the
  cyclotomic  Brauer  algebras in characteristic zero} 
\author{C.~Bowman}
\email{Bowman@math.jussieu.fr}
\address{Institut de MathŽmatiques de Jussieu, 175 rue du chevaleret,
  75 013, Paris} 
\author{A.~G.~Cox} 
\email{A.G.Cox@city.ac.uk} 
\author{M.~De Visscher}
\email{Maud.Devisscher.1@city.ac.uk }
\address{Centre for Mathematical Science,
 City University London,
 Northampton Square,
 London,
 EC1V 0HB,
 England.}
\subjclass[2000]{20C30} \date{\today}

\begin{abstract}
We study the representation theory of the cyclotomic Brauer algebra
via truncation to idempotent subalgebras which are isomorphic to a
product of walled and classical Brauer algebras.  In particular, we
determine the block structure and decomposition numbers in
characteristic zero.
\end{abstract}
\maketitle

\section*{Introduction} 
The symmetric and general linear groups satisfy a double centraliser
property over tensor space.  This relationship is known as Schur--Weyl
duality and allows one to pass information between the representation
theories of these algebras. The Brauer algebra is an enlargement of
the symmetric group algebra and is in Schur-Weyl duality with the
orthogonal (or symplectic) group.

The cyclotomic Brauer algebra $B^m_n$ is a corresponding enlargement
of the complex reflection group algebra $H_n^m$ of type $G(m,1,n)$.
This was introduced by \cite{hoBMW} as a specialisation of the
cyclotomic BMW algebra, and has been studied by various authors (see
for example \cite{ghm, ruixu,ruiyu, yu}).

The algebra $H_n^m$ is Morita equivalent to a direct sum of products
of symmetric group algebras. One might ask if this equivalence extends
to the cyclotomic Brauer algebra. Although there is no direct
equivalence, we will see that the underlying combinatorics of $B^m_n$
is that of a product of classical Brauer and walled Brauer algebras.

Our main result is that certain co-saturated idempotent subalgebras of
$B^m_n$ are isomorphic to a product of classical Brauer and walled
Brauer algebras. Over a field of characteristic zero, this induces
isomorphisms between all higher extension groups
$\Ext^i(\mathcal{F}(\Delta), -)$.  Hence we obtain the decomposition
numbers and block structure of the cyclotomic Brauer algebra in
characteristic zero from the corresponding results for the Brauer and
walled Brauer algebras \cite{marbrauer, cdv}.

We exhibit a tower of recollement structure \cite{cmpx} for $B_n^m$, and
discuss certain signed induction and restriction functors associated
with this. We expect that this structure will also be a useful tool in the
positive characteristic case.
     
Diagrams for the cyclotomic Brauer algebra come with an orientation
due to the relationship with the cyclotomic BMW algebra. However,
one can define a similar algebra without orientation, which we shall
call the unoriented cyclotomic Brauer algebra. In an Appendix we show
that our results can be easily modified for this algebra, to reduce
its study to a product now just of Brauer algebras. The advantage of
this unoriented version is that analogues can be defined associated to
general complex reflection groups of type $G(m,p,n)$; we will consider
the representation theory of such algebras in a subsequent paper.

\section{Cyclotomic Brauer algebras}\label{defintionsandstuff}

In this section we define the cyclotomic Brauer algebra,
$B_n^m=B^m_n(\delta)$ over an algebraically closed field $k$ of
characteristic $p\geq 0$. We assume throughout the paper that $m$ is
invertible in $k$ and we fix a primitive $m$-th root of unity $\xi$.
 
\subsection{}\textbf{Definitions}
 
Given $m, n\in \mathbb{N}$ and $\delta=(\delta_0, \ldots,\delta_{m-1})
\in k^{ m }$, the {\it cyclotomic Brauer algebra} $B_n^m(\delta)$ is a
finite dimensional associative $k$-algebra spanned by certain Brauer
diagrams.  An {\it $(m,n)$-diagram} consists of a {\it frame} with $n$
distinguished points on the northern and southern boundaries, which we
call {\it nodes}.  We number the northern nodes from left
  to right by $1\ldots n$ and the southern nodes similarly by
  $\bar{1},\ldots,\bar{n}$. Each node is joined to precisely one
other by a strand; strands connecting the northern and southern edge
will be called \emph{through strands} and the remainder \emph{arcs}.
There may also be closed loops inside the frame, those diagrams
without closed loops are called \emph{reduced} diagrams.

Each strand is endowed with an orientation and labelled by an element
of the cyclic group $\ZZ/m\ZZ$.  We may reverse the orientation by
relabelling the strand with the inverse element in $\ZZ/m\ZZ$.  We
identify diagrams in which the strands connect the same pairs of nodes
and  (after being identically oriented) have the same labels.

As a vector space, $B^m_n$ is the $k$-span of all reduced
$(m,n)$-diagrams.  Figure \ref{xyelts} gives an example of two such
elements in $B_{6}^3(\delta)$.

\begin{figure}[ht]
\includegraphics[width=8cm]{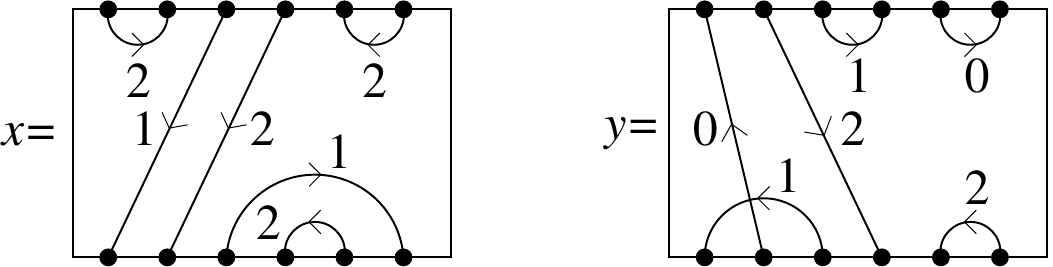}
\caption{Two elements in $B_6^3(\delta)$.}
\label{xyelts}
\end{figure}

We define the product $x \cdot y$ of two reduced $(m,n)$-diagrams $x$
and $y$ using the concatenation of $x$ above $y$, where we identify
the southern nodes of $x$ with the northern nodes of $y$.  More
precisely, we first choose compatible orientations of the strands of
$x$ and $y$. Then we concatenate the diagrams and add the labels on
each strand of the new diagram to obtain another $(m,n)$-diagram.

Any closed loop in this $(m,n)$-diagram can be oriented such that as
the strand passes through the leftmost central node in the loop it
points downwards. If this oriented loop is labelled by $i\in\ZZ/m\ZZ$
then the diagram is set equal to $\delta_i$ times the same diagram
with the loop removed.

\begin{eg} 
Consider the product $x \cdot y$ of the elements in Figure
\ref{xyelts}. After concatenation we obtain the element in Figure
\ref{product1}. Reading from left to right in the diagram we have that
$1-0 \equiv 1$, $2+2\equiv1$, and $1-2-1+0\equiv1$, ($\text{mod } 3$)
and therefore we obtain the
reduced diagram in Figure \ref{productfinal} by
removing the closed loop labelled by 1, and multiply by $\delta_1$.
\end{eg}

\begin{figure}[ht]
\includegraphics[width=3.0cm]{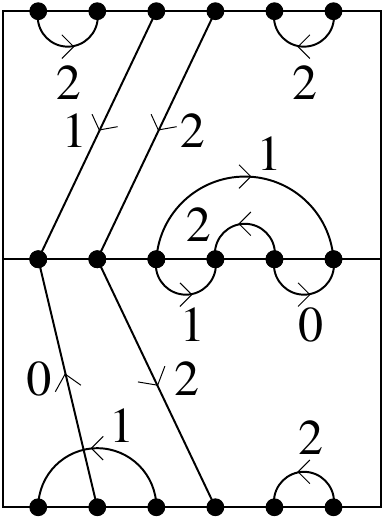}
\caption{The unoriented and unreduced product of $x$ and $y$.}
\label{product1}
\end{figure}

\begin{figure}[ht]
\includegraphics[width=3.9cm]{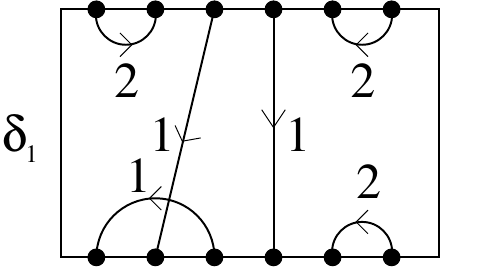}
\caption{The reduced product of $x$ and $y$.}
\label{productfinal}
\end{figure}

From now on, we will omit the label on any strand labelled by $0\in
\ZZ /m\ZZ$.

We will need to speak of certain elements of the algebra with great
frequency. Let $t_{i,j}$ (for $1\leq i,j\leq n$) be the diagram with
only $0$ labels and having through strands from $i$ to $\bar{j}$, $j$
to $\bar{i}$, and $l$ to $\bar{l}$ for all $l\neq i,j$. Let $t_i^r$
(for $1\leq i\leq n$ and $0\leq r\leq m-1$) be the diagram with
through strands from $l$ to $\bar{l}$ for all $l$, with the through
strand from $i$ labelled by $r$ and all other labels being $0$. These
elements are illustrated in Figure \ref{tare}. We let $e_{i,j}$ (for
$1\leq i,j\leq n$) be the diagram with only $0$ labels and having arcs
from $i$ to $j$ and $\bar{i}$ to $\bar{j}$, and through strands from
$l$ to $\bar{l}$ for all $l\neq i,j$. This element is illustrated in
Figure \ref{eare}.

The elements $t_{i,i+1}$ (with $i\leq n-1$) and $t_{1}^1$ are
generators of the group algebra 
$$H^m_n=k((\ZZ/m\ZZ) \wr \Sigma_n)$$ as a subalgebra of
$B^m_n(\delta)$. Note that $B^m_1(\delta)\cong k(\ZZ/m\ZZ)$; for
convenience, we set $B^m_0(\delta)=k$. It is easy to see that the
cyclotomic Brauer algebra is generated by the elements $t_{i,i+1},
t_1^1$, and $e_{1,2}$ (for $i\leq n-1$).

\begin{figure}[ht]
\includegraphics[width=7.0cm]{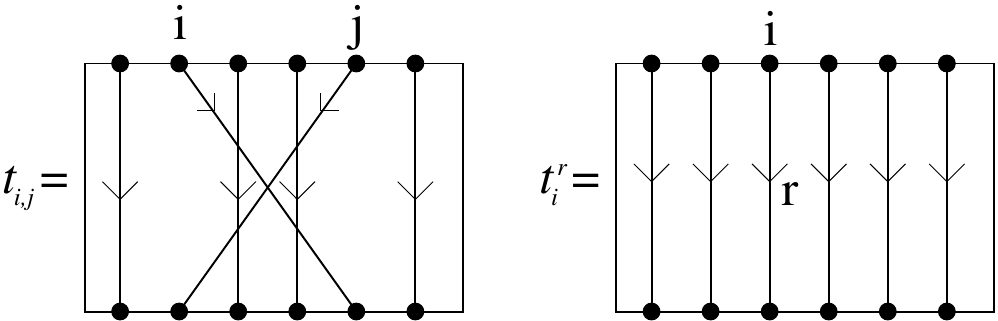}
\caption{The elements $t_{i,j}$ and $t_i^r$.}
\label{tare}
\end{figure}

\begin{figure}[ht]
\includegraphics[width=3.0cm]{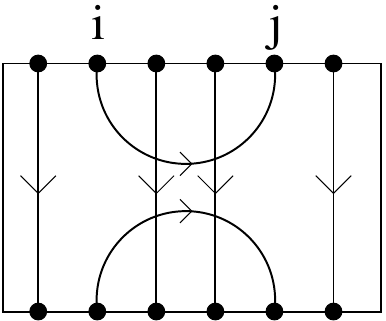}
\caption{The element $e_{i,j}$.}
\label{eare}
\end{figure}

We have defined the cyclotomic Brauer algebra in terms of $\delta
=(\delta_0, \ldots, \delta_{m-1}) \in k^{ m}$.  However, we shall find
that it is signed polynomials in these parameters which govern the
representation theory of the algebra.

\begin{defn}
For each $0\leq r \leq m-1$ we define the $r$th signed cyclotomic
parameter to be
\begin{align*}
\overline{\delta}_r = 
\frac{1}{m} \sum_{i=0}^{m-1} \xi^{ir} \delta_i.
\end{align*}
Note that $\overline{\delta}_r$ and $\overline{\delta}_{m-r}$ are
swapped by the map $\xi \leftrightarrow \xi^{-1}$.
\end{defn}

\begin{rem}
Cyclotomic Brauer algebras were originally defined by
H\"aring-Oldenburg \cite{hoBMW}.  Our definition can easily be seen to
be equivalent to that of Rui and his collaborators (see \cite{ruiyu}
and \cite{ruixu}).  The version considered by Goodman and Hauschild
Mosley \cite{ghm} and Yu \cite{yu} is the specialisation of this
algebra obtained by setting $\delta_r = \delta_{m-r}$.
\end{rem}

\begin{rem}
In \cite{ruiyu} and \cite{ruixu} semi-simplicity conditions are given
for $B^m_n$ in terms of the signed parameters.  We note that there is
a mistake in the statement of \cite[Theorem 8.6]{ruiyu} which runs
through both of these papers.  This is a simple misreading of the
(correct) circulant matrices calculated in the proof of the theorem.

Their vanishing conditions are given in terms of $\overline{\delta}_r
- m\epsilon_{(r,0)}$ where $\epsilon_{(r,0)}$ is the Kronecker
function.  Correct versions of these statements can be deduced by
substituting this  by $\overline{\delta}_r - m\epsilon_{(r,m-r)}$.
Compare \cite[Theorem 6.2]{cddm} and \cite[Proposition 4.2]{cdm} to
see how the $-m\epsilon_{(r,m-r)}$ relates to the semi-simplicity of
the Brauer algebra versus the walled Brauer algebra.
\end{rem}

\begin{rem}
We have that $B_{n}^2(\delta)$ is a subalgebra of the recently defined
Brauer algebra of type $C_n$ (see \cite{cly}).  This can be seen by
`unfolding' the diagrams (as outlined in \cite[Section 4.3]{mgp}) and
using \cite[Theorem 3.6]{bowbrauer}.
\end{rem}

\subsection{}\textbf{Classical Brauer and walled Brauer algebras}

The classical Brauer algebra $B(n,\delta)$ ($\delta\in k$) is given by
the particular case $B_n^1(\delta)$ with $\delta_0=\delta$. Note that
the orientation of the strands in Brauer diagrams plays no role in
this case and so can be ignored.

The walled Brauer algebra $WB(r,s,\delta)$ is the subalgebra of
$B(r+s, \delta)$ spanned by the so-called walled Brauer
diagrams. Explicitly, we place a vertical wall in the (r+s)-Brauer
diagrams after the first $r$ northern (resp. southern) nodes and we
require that arcs must cross the wall and through strands cannot cross
the wall.

\section{Representations of $H^m_n$}
In this section we review the construction of the Specht modules
for the  group algebra $H^m_n$ of the complex reflection group
$(\ZZ/m\ZZ) \wr  \Sigma_n$.

\subsection{}\textbf{Compositions and partitions}\label{comb}

An $m$-{\it composition} of $n$ is an $m$-tuple of non-negative
integers $\omega=(\omega_0, \ldots , \omega_{m-1})$ such that
$\sum_{i=0}^{m-1}\omega_i = n$. A partition is a finite decreasing
sequence of non-negative integers. An $m$-{\it partition} of $n$ is an
$m$-tuple of partitions $\lambda =(\lambda^0, \ldots , \lambda^{m-1})$
such that $\sum_{i=0}^{m-1} |\lambda^i|=n$ (where $|\lambda^i|$
denotes the sum of the parts of the partition $\lambda^i$).  Given an
$m$-partition $\lambda$ we associate the $m$-composition
$$|\lambda|=(|\lambda^0|,|\lambda^1|, \ldots , |\lambda^{m-1}|).$$ For
an $m$-composition of $n$, $\omega$, we define another $m$-composition
$[\omega]$ by
$$[\omega]=([\omega_0], [\omega_1], \ldots, [\omega_{m-1}]=n)$$
where $[\omega_r]=\sum_{i=0}^r \omega_i$ for $0\leq r\leq m-1$.
For an $m$-partition $\lambda$ we define $[\lambda]=[|\lambda|]$.

The \textit{Young diagram} of an $m$-partition is simply the $m$-tuple
of Young diagrams of each partitions. We do not distinguish between
the $m$-partition $\lambda$ and its Young diagram.  For an
$m$-partition $\lambda$, define the set $\remb(\lambda)$
(resp.$\add(\lambda)$) of all \emph{removable boxes} (respectively
\emph{addable boxes}) to be those which can be removed from
(respectively added to) ${\lambda}$ such that the result is the Young
diagram of an $m$-partition.  We can refine this by insisting that a
removable (respectively addable) box has sign $\xi^r$ if it can be
removed (respectively added) to $\lambda^r$, for $0\leq r \leq m-1$.
We denote these sets by $\xi^r$-$\remb(\lambda)$ and
$\xi^r$-$\add(\lambda)$ respectively.

\subsection{}\textbf{Idempotents}

We define some idempotents in $H_n^m$ which play a very important role
in this paper.  Note that $k(\ZZ /m\ZZ \times \ldots \times \ZZ/m\ZZ)$
occurs naturally as the subalgebra of $H_n^m$ spanned by all diagrams
where node $i$ is connected to node $\overline{i}$ for all $1\leq
i\leq n$. As $m$ is invertible in $k$ we have that $k(\ZZ/m\ZZ)$ is
semisimple and decomposes into a sum of 1-dimensional modules given by
$\xi^r$ ($0\leq r\leq m-1$). We denote by $T_i^r$ the idempotent in
the copy of $k(\ZZ/m\ZZ)$ on the i-th strand corresponding to
$\xi^r$. This idempotent is given as follows.

\begin{defn}\label{Ti} For each $1\leq i\leq n$ and each $0\leq r\leq
  m-1$, define the idempotent 
$$T_i^r= \frac{1}{m} \sum_{1 \leq q \leq m}\xi^{qr}t_i^q.$$
\end{defn}

Now we will consider certain products of these idempotents.
 Let $\omega$ be an $m$-composition of $n$. We have
$$0\leq [\omega_0]\leq [\omega_1]\leq [\omega_2]\leq \ldots \leq
[\omega_{m-1}]=n.$$ So for each $1\leq i\leq n$ there is a unique
$0\leq r\leq m-1$ with
$$[\omega_{r-1}]< i \leq [\omega_r]$$ (where we set
$[\omega_{-1}]=0$). In this case we write $i\in [\omega_r]$. 
Now we define the idempotent $\pi_\omega$ as follows.

\begin{defn}\label{piomega} Let $\omega$ be an $m$-composition of
  $n$. Then we define 
$$\pi_\omega = \prod_{r=0}^{m-1} \prod_{i\in [\omega_r]}T_i^r.$$
\end{defn}

The element $\pi_\omega$ is a linear combination of diagrams, but can
be viewed as putting the element $T^0$ on each of the first $\omega_0$
strands of the identity diagram, then the element $T^1$ on each of the
next $\omega_1$ strands,..., and finally $T^{m-1}$ on each of the last
$\omega_{m-1}$ strands.

\subsection{}\textbf{Specht modules of $H^m_n$}\label{spechth} 

For an $m$-composition $\omega=(\omega_0, \omega_1, \ldots ,
\omega_{m-1})$ of $n$ we define the Young subgroup $\Sigma_\omega$ of
$\Sigma_n$ by
$$\Sigma_\omega =\Sigma_{\omega_0}\times \Sigma_{\omega_1}\times
\ldots \times \Sigma_{\omega_{m-1}}$$ and the corresponding Young
subalgebra $H_\omega^m$ of $H_n^m$ by
$$H_\omega^m=k((\ZZ/m\ZZ) \wr \Sigma_\omega).$$

 \begin{defn}
Let $\lambda, \mu$ be $m$-partitions of $n$.  We say that $\lambda$
dominates $\mu$ and write $\mu \unlhd_n \lambda$ if
\begin{align*}
[\lambda^{j-1}] + \sum_{i=1}^k \lambda^j_i \geq [\mu^{j-1}] +
\sum_{i=1}^k \mu^j_i
\end{align*}
for all $0 \leq j \leq m-1$ and $k\geq 0$ (where we set
$[\lambda^{-1}]=[\mu^{-1}]=0$).   
 \end{defn}
 
Given any $k\Sigma_n$-module $M$ and any $r\in \ZZ/m\ZZ$ we define the
$H_n^m$-module $M^{(r)}$ by setting $M^{(r)}\!\downarrow_{\Sigma_n}=M$
and each $t_i$ ($1\leq i\leq n$) acts on $M^{(r)}$ by scalar
multiplication by $\xi^r$. In particular, if $\lambda$ is a partition
of $n$ and we denote by $S(\lambda)$ the corresponding Specht module
for $k\Sigma_n$ then we have an $H_n^m$-module $S(\lambda)^{(r)}$ for
each $0\leq r\leq m-1$. This module is the Specht $H_n^m$-module
labelled by $(\emptyset, \ldots, \emptyset ,\lambda, \emptyset, \ldots
,\emptyset)$ where $\lambda$ is in the $r$-th position. More generally
we have the following result.

 \begin{prop}[Section 5 of \cite{gl}]
The algebra $H^m_n$ is cellular with respect to the dominance order
$\unlhd_n$ on the set of $m$-partitions of $n$.  For a given
$m$-partition $\lambda$ of $n$, the cell module $\mathbf{S}(\lambda)$
is given by
$$\mathbf{S}(\lambda) \cong (S(\lambda^0)^{(0)} \otimes
\ldots
\otimes S(\lambda^{m-1})^{(m-1)})\!\uparrow_{H^m_{|\lambda|}}^{H^m_n}.$$
\end{prop}

We call $\mathbf{S}(\lambda)$ the Specht module for $H_n^m$ labelled
by $\lambda$.

\medskip

It is well known (see for example \cite{dm}) that the algebra $H_n^m$
is Morita equivalent to the direct sum of group algebras of Young
subgroups of $\Sigma_n$. These arise as idempotent subalgebras of
$H_n^m$.  Indeed, the idempotent subalgebra $\pi_\omega
H_n^m\pi_\omega$ is isomorphic to $k\Sigma_\omega$ and under this
isomorphism we have
\begin{equation}\label{omegaspecht}
\pi_\omega \mathbf{S}(\lambda)\cong \left\{\begin{array}{ll}
S(\lambda^0)\otimes S(\lambda^1)\otimes \ldots \otimes
S(\lambda^{m-1}) & \mbox{if $|\lambda|=\omega$}\\ 0 &
\mbox{otherwise.}\end{array}\right.
\end{equation}

\section{Cell modules for $B^m_n$}\label{section3} 

In this section we show that $B_n^m$ is an iterated inflation (in the
sense of \cite{kxbrauer}), and so is a cellular algebra. We recall the
construction of the cell modules and study the restriction and
induction rules for these. When the algebras are
  quasi-hereditary we obtain a tower of recollement (in the sense of
  \cite{cmpx}). 

\subsection{}\textbf{Iterated inflation and cell modules}

\begin{defn}  Suppose that $n,l\in\NN$ with $l\leq
  \lfloor n/2\rfloor$.  An $(n,l)$-dangle is a partition of $\{1 ,
  \ldots, n\}$ into $l$ two-element subsets (called arcs) and $n-2l$
  one-element subsets (called free nodes).  An $(m,n,l)$-dangle is an
  $(n,l)$-dangle to which an integer $r \in \ZZ/m\ZZ$ has been
  assigned to every subset of size 2.
\end{defn}

We can represent an $(n,l)$-dangle $d$ by a set of $n$ nodes labelled
by the set $\{1, \ldots, n\}$, where there is an arc (denoted
$v_{ij}$) joining $i$ to $j$ if $\{i,j\} \in d$, and there is a
vertical line starting from $i$ if $\{i\} \in d$.  An $(m,n,l)$-dangle
can be represented graphically by first labelling each
arc of the underlying $(n,l)$-dangle and then giving it the following
orientation: we let all one element sets have a downward orientation
and all two element sets have a right orientation.  An example of an
$(m,7,3)$ for $m\geq 3$ is given in Figure \ref{dangle}.

\begin{figure}[ht]
\centerline{
\begin{minipage}{54mm}
\def\objectstyle{\scriptstyle}
\xymatrix@=2pt{
\circ  \ar@{->}[ddd]		&&\circ  \ar@{->}[ddd]
&&\circ\ar@{->}@/_.5pc/[rrrr]|\hole|{\textbf{2}}|\hole	&&\circ
\ar@{->}@/_.5pc/[rrrr]|\hole|{\textbf{1}}|\hole		&&\circ
&&\circ	&&\circ \ar@{->}[ddd]   \\ 
&&&&&&&&&&&\\	
&&&&&&&&&&&&\\
&&&&&&&&&&&&\\
}
\end{minipage}}
\caption{An $(m,7,3)$ dangle.}
\label{dangle}
\end{figure}
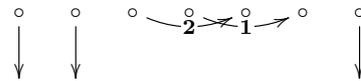

We let $V(m,n,l)$ denote the vector space spanned by all $(m,n,l)$-dangles.  

Let $\Lambda(m,n)$ denote the set of $m$-partitions of $n-2l$ for all
$l \leq \lfloor n/2\rfloor$.  The dominance ordering extends naturally
to this set by writing $\lambda \unrhd \mu$ if and only if either
$\sum_{i=0}^{m-1}|\lambda^i| > \sum_{i=0}^{m-1}|\mu^i|$ or
$\sum_{i=0}^{m-1}|\lambda^i| = \sum_{i=0}^{m-1}|\mu^i|=n-2l$ (for some
$l$) and $\mu {\unlhd_{n-2l}} \lambda$.

Each $(m,n)$-diagram in $B_n^m$ with $n-2l$ through strands can be
decomposed as two $(m,n,l)$-dangles, giving the top and bottom of the
diagram, and an element of $\ZZ/m\ZZ\wr \Sigma_{n-2l}$ giving the
through strands. Using this decomposition we get the following result.

\begin{thm}\label{cellstrattypec}
The cyclotomic Brauer algebra $B_{n}^m(\delta)$ is an iterated
inflation with inflation decomposition
\begin{align*}
B_{n}^m(\delta) = \bigoplus^{\lfloor n/2\rfloor}_{l=0} V(m,n,l)
\otimes V(m,n,l) \otimes H^m_{n-2l}. 
\end{align*}
Therefore $B^m_n$ is cellular with respect to the dominance ordering
on $\Lambda(m,n)$, and the anti-involution $\ast$ given by reflection
of a diagram through its horizontal axis.
\end{thm}
\begin{proof}
This follows by standard arguments, see for example
\cite{kxbrauer}.
\end{proof}

For any $m$-partition $\lambda$ of $n-2l$, we define an action of the
algebra $B_n^m$ on the vector space $V(m,n,l) \otimes
\mathbf{S}(\lambda)$. For any $(m,n)$-diagram $X$, any
$(m,n,l)$-dangle $d$ and any element $x\in \mathbf{S}(\lambda)$ we
define $X(d\otimes x)$ as follows. Place the diagram $X$ above the
$(m,n,l)$-dangle $d$. Choose a compatible orientation of the strands
and then concatenate to give an $(m,n,l+t)$-dangle $Xd$ and an element
$\sigma\in H^m_n$ acting on the free $n-2(l+t)$ nodes. If $t>0$ then
we set $X(d\otimes x)=0$, and otherwise, we define $X(d\otimes
x)=(Xd)\otimes \sigma x$.

\begin{cor}
We have that the cell modules for $B_{n}^m(\delta)$ are of the form
\begin{align*}
 \Delta_n (\lambda) = V(m,n,l) \otimes 
\mathbf{S}(\lambda) 
\end{align*}
where $\mathbf{S}(\lambda)$ is the Specht module for $H^m_{n-2l}$
defined in Section 2.3.
\end{cor}

\subsection{}\textbf{Tower of algebras, restriction and induction} 

Let $n\geq 2$. Suppose first that $\delta\neq 0\in k^{m}$ and fix a
$\delta_r \neq 0$ for some $0\leq r\leq m-1$. We then define the
idempotent $e_{n-2}=\frac{1}{\delta_r}t_{n-1}^re_{n-1,n}$ 
as illustrated in Figure \ref{idemis}. Note that
it is a scalar multiple of a diagram with $n-2$ through strands. 
If $\delta =0$ and $n\geq 3$ then we define $e_{n-2}$ to be the
diagram with strands given by $\{i, \overline{i}\}$ for all $1\leq
i\leq n-3$, $\{n-2, \overline{n}\}$, $\{n-1,n\}$ and
$\{\overline{n-2}, \overline{n-1}\}$, as illustrated in
  Figure \ref{altidem}.

\begin{figure}[ht]
\includegraphics[width=4.0cm]{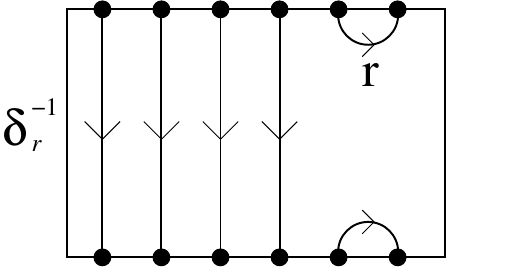}
\caption{The idempotent $e_{n-2}$ when $\delta_r\neq 0$.}
\label{idemis}
\end{figure}

\begin{figure}[ht]
\includegraphics[width=3.0cm]{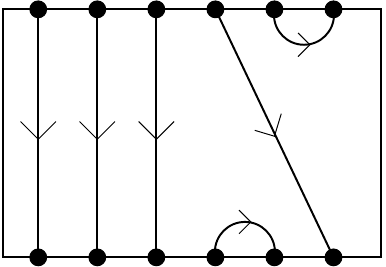}
\caption{The idempotent $e_{n-2}$ when $\delta=0$ and $n\geq 3$.}
\label{altidem}
\end{figure}

It is easy to see that 
\begin{equation}\label{eAe}
e_{n-2}B_n^me_{n-2}\cong B_{n-2}^m
\end{equation} 
and
\begin{equation}\label{AeA}
B_n^m/B_n^m e_{n-2}B_n^m \cong H_n^m
\end{equation}
just as for the Brauer algebra \cite[Lemma 2.1]{cdm}. In particular,
any $H_n^m$-module can naturally be viewed as a $B_n^m$-module.

Via (\ref{eAe}) we define 
an exact localisation functor through the idempotent
$e_{n-2}$.  
\begin{align*} F_n \, :\, B_{n}^m\mbox{\rm -mod}
  &\longrightarrow B_{n-2}^m\mbox{\rm -mod}\\ M &\longmapsto e_{n-2} M
  \intertext{and a right exact globalisation functor} G_n\, : \,
  B_{n}^m\mbox{\rm -mod} &\longrightarrow B_{n+2}^m\mbox{\rm -mod}\\ M
  &\longmapsto B_{n+2}^me_{n} \otimes_{B_{n}^m}M.
\end{align*}
Note that $F_{n+2}G_n(M)\cong M$ for all $M\in B_{n}^m\mbox{\rm
  -mod}$, and hence $G_n$ is a full embedding. It is easy to check
that for any $\lambda\in \Lambda(m,n)$ we have  
\begin{equation}\label{localisation}
F_n(\Delta_n(\lambda)) \cong \left\{ \begin{array}{ll}
  \Delta_{n-2}(\lambda) & \mbox{if $\lambda \in \Lambda(m,n-2)$}\\ 0 &
  \mbox{otherwise}.\end{array}\right.
\end{equation}
We have, for any
$H_n^m$-module, that
$$B_{n+2}^me_n\otimes_{B_n^m}M\, \cong \, V(m,n+2,1)\otimes_k M$$ (see
\cite[Proposition 4.1]{hhkp}). In particular, if
$\lambda$ is an $m$-partition of $n-2l$, we have
$$\Delta_n(\lambda) = G_{n-2}G_{n-4}\ldots G_{n-2l}\mathbf{S}(\lambda)$$
and hence
\begin{equation}\label{GDelta}
G_n\Delta_n(\lambda)=\Delta_{n+2}(\lambda).
\end{equation}

\begin{lem}\label{A3} For each $n\geq 1$, the algebra
  $B^m_{ n}$ can be identified as a subalgebra of
  $B^m_{ n+1}$ via the homomorphism which takes an $(m,n)$-diagram
  $X$ in $B^m_{ n}$ to the $(m,n+1)$-diagram in $B^m_{ n+1}$
  obtained by adding two vertices $n+1$ and
  $\overline{n+1}$ with a strand between them labelled by zero. 
 \end{lem}

Lemma \ref{A3} implies that we can consider the usual restriction and
induction functors.  We refine
these functors as a direct sum of signed versions.  This refinement is
given by inducing via
 \begin{align*}
 &B^m_{n-1} \subset B^m_{n-1} \otimes B_1^m \subset B_n^m,
 \end{align*}
where $B_1^m \cong k(\ZZ/m\ZZ)$  and corresponds to the
  rightmost string in the diagrams. Note that for any $B_n^m$-module $M$
we have that $T_n^rM$ is naturally a $B_{n-1}^m$-module. In fact, it
is given by the summand of $M\!\!\downarrow_{B_{n-1}^m\otimes B_1^m}$ on
which $B_1^m$ acts by $\xi^r$. So we can define the following signed
induction and restriction functors.
\begin{align*} 
\xi^r\!\text{-}\!\res_n \, :\, &B^m_{ n}\mbox{\rm -mod}
\longrightarrow  B^m_{ n-1}\mbox{\rm -mod}\\ 
& M \longmapsto T_n^rM\!\downarrow_{B^m_{n-1}}
\intertext{
and  }
\xi^r\!\text{-}\!\ind_n\, : \,  &B^m_{ n}\mbox{\rm -mod}
\longrightarrow  B^m_{ n+1}\mbox{\rm -mod}\\ 
&M \longmapsto  \ind^{ B^m_{ n} }_{ B^m_{ n-1}\otimes B^m_1} (M
\boxtimes kT_n^r) . 
\end{align*}

We can relate these functors to globalisation and localisation, as in
\cite[(A4)]{cmpx}.
 
\begin{lem}\label{A4} (i) For all $n\geq 2$ we have that 
$$B^m_{ n}e_{n-2} \cong  B^m_{ n-1}$$ as a left $B^m_{ n-1}$, right
  $B^m_{ n-2}$-bimodule.\\
(ii) For all $B^m_{ n}$-modules
$M$ we have $$\xi^r\!\text{-}\!\res_{n+2}(G_n(M))\cong
\xi^{m-r}\!\text{-}\!\ind_n(M).$$ 
\end{lem}
\begin{proof}
 (i) Every diagram in $B^m_{n}e_{n-2}$ has an edge between
  $\overline{n-1}$ and $\bar{n}$. Define a map from $B^m_{n}e_{n-2}$
  to $B^m_{n-1}$ by sending a diagram $X$ to the diagram with $2(n-1)$
  vertices obtained from $X$ by removing the line connecting
  $\overline{n-1}$ and $\bar{n}$ and the line from $n$ (labelled by
  $r$), and pairing the vertex $\overline{n-1}$ to
  the vertex originally paired with $n$ in $X$ (labelling this line
  with $r$ and preserving the orientation). It is easy to
  check that this gives an isomorphism.\\ 
(ii) We have to show that 
$$T_n^rB_n^me_{n-2}\otimes_{B_{n-2}^m}M \, \cong \,
  B_{n-1}^m\otimes_{B_{n-2}^m\otimes B_1^m}(M\boxtimes kT_{n-1}^{m-r}).$$
  The left hand side is spanned by all elements obtained from diagrams
  in $B_n^me_{n-2}$ by attaching the idempotent $T^r$ to node
  $n$. Following the map given in (i) gives the required isomorphism.
\end{proof}

\begin{rem}
Note that this construction will generalise to any tower of
recollement (as in \cite{cmpx}) where $A_{n-1}\otimes
A_1\subset A_n$, and $A_1$ admits a (non-trivial) direct sum decomposition.
\end{rem}

Given a family of modules $M_i$ we will write $\biguplus_i M_i$ to
denote some module with a filtration whose quotients are exactly the
$M_i$, each with multiplicity one. This is not uniquely defined as a
module, but the existence of a module with such a filtration will be
sufficient for our purposes.  
 
\begin{prop}\label{indres} 
(i) For  
$\lambda \in \Lambda({m, n})$ we have short exact sequences
$$0\longrightarrow  \!\!\!\!    \biguplus_{\square \in
    {\xi^{m-r}}\!\text{-}\!\remb(\lambda)}   \!\!\!\!
  \Delta_{n+1}(\lambda - \square) 
\longrightarrow {\xi^r}\!\text{-}\!\ind_n\, \Delta_n(\lambda)
\longrightarrow  \!\!\!\!    \biguplus_{\square \in
  {\xi^r}\!\text{-}\!\add(\lambda)}  \!\!\!\!   \Delta_{n+1}(\lambda +
\square)\longrightarrow 0$$ 
 and
$$0\longrightarrow  \!\!\!\!    \biguplus_{\square \in
  {\xi^{r}}\!\text{-}\!\remb(\lambda)}   \!\!\!\!    \Delta_{n-1}(\lambda
- \square) 
\longrightarrow {\xi^r}\!\text{-}\!\res_n\, \Delta_n(\lambda)
\longrightarrow  \!\!\!\!    \biguplus_{\square \in
  {\xi^{m-r}}\!\text{-}\!\add(\lambda)}  \!\!\!\!   \Delta_{n-1}(\lambda +
\square)\longrightarrow 0.$$
(ii) In each of the filtered modules which arise in (i), the
filtration can be chosen so that partitions labelling successive
quotients are ordered by dominance, with the top quotient maximal
among these. When $H^m_n$ is semisimple the $\biguplus$ all become
direct sums.
\end{prop}
 
\begin{proof} We prove the result for the functor $\xi^r\!\text{-}\!\res_n$.
The result for $\xi^r\!\text{-}\!\ind_n$ then follows immediately from
Proposition \ref{A4}(ii) and (\ref{GDelta}).

We let $W$ be the
subspace of $\xi^r\!\text{-}\!\res_n\, \Delta_n(\lambda)$ spanned by
all elements of the form $T_n^rd\otimes x$ with $d\in V(m,n,l), x\in
\mathbf{S}(\lambda)$ such that the node $n$ is free in $d$.  It is
clear that this subspace is a $B^m_{ n-1}(\delta)$-submodule.  We
shall prove that $W= \biguplus_{\square \in
  {\xi^{r}}\!\text{-}\!\remb(\lambda)}\Delta_{n-1}(\lambda - \square)
$.

By the restriction rules for cyclotomic Hecke algebras (see
\cite{matspecht} for details), it will be enough to show that 
$$W\cong V(m,n-1,l) \otimes
\xi^r\!\text{-}\!\res_{n-2l}\mathbf{S}(\lambda)$$ where
$\xi^r\!\text{-}\!\res_{n-2l}\mathbf{S}(\lambda)=T_{n-2l}^r\mathbf{S}(\lambda)$
viewed as a $H_{n-2l-1}^m$-module. The map sending $T_n^rd\otimes x$
to $\phi(d)\otimes T_{n-2l}^rx$ where $\phi(d)$ is the
$(m,n-1,l)$-dangle obtained from $d$ by removing node $n$ is clearly
an isomorphism.

We will now show that 
\begin{align*}
U= T_n^r\Delta(\lambda)/W \cong V(m,n-1,l-1)\otimes
(\ind_{H_{n-2l}^m\otimes
  H_1^m}^{H_{n+1-2l}^m}(\mathbf{S}(\lambda)\boxtimes kT^{m-r}))
\end{align*}
which gives the required result using \cite{matspecht}.  Let $d$
be an $(m,n,l)$-dangle which has an arc from node $n$ to some other
node. Number the free vertices of $d$ and the node connected to $n$ in
order from left to right with the integers $1,\ldots n+1-2l$. Say that
the node connected to $n$ is numbered with $i$. Define $\psi(d)$ to be
the $(m,n-1,l-1)$-dangle obtained from $d$ by removing the arc
$\{i,n\}$ and deleting the node $n$ (so that $i$ becomes a free
node). And define the permutation $\sigma_i=(i,n-2l+1,
n-2l,n-2l-1,\ldots i+1)\in \Sigma_{n-2l+1}$.  This element is obtained
by pulling down node $n$ in $d$ (as in the proof of Proposition
\ref{A4}(i)), giving the permutation $\sigma_i$ of the free vertices
$\{1,2,\ldots , n-2l+1\}$.  Now the map sending $T_n^rd\otimes x$ to
$\psi(d)\otimes (\sigma_i \otimes_{H_{n-2l}^m\otimes
  H_1^m} (x\otimes T_{n-2l+1}^{m-r})$ gives the required
isomorphism.
\end{proof} 

We have seen that the induction and restriction
functors for $B^m_n$ decompose into signed versions.  We saw in Lemma
\ref{A4} and Proposition \ref{indres} that when $r\neq m-r$, the
$\xi^r$- and $\xi^{m-r}$-functors `pair-off' in a manner reminiscent
of those for the walled Brauer algebra.  In the case that $r=m-r$ we
saw that the $\xi^{r}$-functors behave like those of the classical
Brauer algebra. We will make this connection explicit in Section 5.

\subsection{}\textbf{Quasi-heredity}

A cellular algebra is quasi-hereditary if and only if it has the same
number of cell modules and simple modules, up to isomorphism.  Using
this and standard arguments for iterated inflations \cite{KXwhen}, we
deduce the following.

\begin{thm}\label{A2'}
Let $k$ be a field of characteristic $p\geq 0$, $m,n\in\NN$, and
$\delta \in k^m$.  If $n$ is even suppose $\delta \neq 0 \in k^m$.
The algebra $B^m_{n}(\delta)$ is quasi-hereditary if and only if
$p>n$ and $p$ does not divide $m$, or $p=0$. 
\end{thm}

%\textbf{Assumption:} 
\begin{ass}\label{qhass}
From now on, we will assume that $\delta\neq 0$
if $n$ is even and that $k$ satisfies the conditions in Theorem
\ref{A2'}, and so $B_n^m(\delta)$ is quasi-hereditary.
\end{ass}

The cell modules $\Delta_n(\lambda)$ are then the \textit{standard
  modules} for this quasi-hereditary algebras, and we will call them
so. Each standard module $\Delta_n(\lambda)$ has simple head
$L_n(\lambda)$ and the set
$$\{L_n(\lambda)\, : \, \lambda\in \Lambda(m,n)\}$$ form a complete
set of non-isomorphic simple $B_n^m$-modules. We denote by
$P_n(\lambda)$ the projective cover of $L_n(\lambda)$.  The results
  in this Section have shown

\begin{thm} Under Assumption \ref{qhass} the algebras $B_n^m(\delta)$
  form a tower of recollement.
\end{thm}

\section{The cyclotomic poset and combinatorics of $B^m_n$}

\subsection{}\textbf{The cyclotomic poset}

Recall that $\Lambda(m,n)$ denotes the set of $m$-partitions of $n-2l$ for
all $l\leq\lfloor n/2\rfloor$.  We let $\Lambda|m,n|$ denote the set
of $m$-compositions of $n-2l$ for all $l\leq \lfloor n/2
\rfloor$. There is a many-to-one map $\Lambda(m,n) 
\to \Lambda|m,n|$ given by
$$(\lambda^0, \ldots, \lambda^{m-1})\mapsto(|\lambda^0|, \ldots,
|\lambda^{m-1}|).$$   For example
 the map $\Lambda(3,9) \to \Lambda|3,9|$ maps $((1^2), (2,1),(2,2))$ to
$(2,3,4)$.
 
For a given $m, n \in \mathbb{N}$ we define a partial ordering
$\preceq$ on $\Lambda|m,n|$.  For $m$-compositions $\omega=(\omega_0,
\omega_1, \ldots, \omega_{m-1})$ and $\omega'=( \omega_0', \omega_1',
\ldots , \omega_{m-1}')\in 
\Lambda|m,n|$, we say $\omega \preceq \omega'$ if and only if
\begin{itemize}
\item[(i)] $\omega_r \leq \omega_r'$ for all $0\leq r\leq m-1$
\item[(ii)] $\omega_r-\omega_r' = \omega_{m-r}-\omega_{m-r}'$ for $r \neq 0, m/2$
\item[(iii)] $\omega_r - \omega_r' \in
2\ZZ$ for $r=0,m/2$.   
\end{itemize}
Any irreducible component in the Hasse diagram of this poset has a
unique minimal element.

Let $\omega \preceq \omega '$, with $a_r=\omega_r -\omega'_r$ for $r\neq
m-r$ and $a_r = (\omega_r-\omega'_r)/2$ for $r=m-r$ (note that $a_r=a_{m-r}$ by
assumption).  We write $\omega \preceq_{\sum a_r\xi^r} \omega'$
where the sum is over all $0\leq r \leq \lfloor m/2\rfloor$.
 
\begin{eg} 
The diagram in Figure \ref{poset1} is an irreducible component of the
Hasse diagram of the poset $(\Lambda|3,6|, \preceq)$.  We have annotated
the edges with the relevant signs.  We see that $(0,0,0)
\preceq_{1+\xi}(1,1,2)$.  Taking the pre-image in $\Lambda(3,6)$, we
see that $(0,0,0) \preceq_{1+\xi} ((1),(1),(2))$ and $(0,0,0)
\preceq_{1+\xi} ((1),(1),(1^2))$.
 \end{eg} 

\begin{figure}[ht]
\centerline{
 \begin{minipage}{54mm}
\def\objectstyle{\scriptstyle}
\xymatrix@=2pt{
(3,3,0)\ar@{-}[ddrr]|\xi  		&&		&&(2,2,2)
  \ar@{-}[ddrr]|\xi  \ar@{-}[ddll]|1 		&&
  &&(1,1,4)	\ar@{-}[ddrr]|\xi  \ar@{-}[ddll]|1 	&&
  &&(0,0,6)	   \ar@{-}[ddll]|1 	\\&\\ 
			&&(2,2,0)	\ar@{-}[ddrr]|\xi	&&
  &&(1,1,2)	\ar@{-}[ddrr]|\xi  \ar@{-}[ddll]|1
  &&&&(0,0,4)\ar@{-}[ddll]|1  \\&\\ 
			&&&&(1,1,0) && &&(0,0,2)  \\&\\
			&&&&&&(0,0,0)\ar@{-}[uurr]|1  \ar@{-}[uull]|\xi
}\end{minipage}
}
\caption{Part of the Hasse poset.}
\label{poset1}
\end{figure}
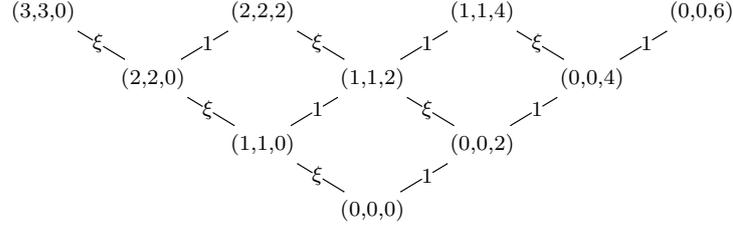
 
\begin{prop}\label{parity}
If $\lambda, \mu \in \Lambda(m,n)$ are such that
$[\Delta_n(\mu):L_n(\lambda)]\neq 0$, then we must have that $|\mu| \preceq
|\lambda|$.  
\end{prop}

\begin{proof} We prove this by induction on $n$. If $n=0$ then
  $B_0^m=k$, therefore there is only one simple module
  $\Delta_0(\emptyset)=L_0(\emptyset)$ and there is nothing to prove. 

Let $n\geq 1$ and suppose that $[\Delta_n(\mu):L_n(\lambda)]\neq 0$,
that is we have a non-zero homomorphism $\Delta_n(\lambda)\rightarrow
\Delta_n(\mu)/N$ for some submodule $N$ of $\Delta_n(\mu)$.  By
localisation, we may assume that $\lambda$ is an $m$-partition of $n$,
so that $L(\lambda)=\Delta(\lambda)$, and that $\mu$ is an
$m$-partition of $n-2l$ for some $l\leq \lfloor n/2\rfloor$.

As $n\geq 1$, $\lambda$ has at least one removable box, $\epsilon$,
say in the $r$th part of the $m$-partition $\lambda$.  Then by
Proposition \ref{indres} (and noting that, under our assumption, each
term is a direct sum) we have that
$$\xi^r\text{-}\ind_{n-1}\Delta_{n-1}(\lambda-\epsilon) \twoheadrightarrow
\Delta_n(\lambda).$$ Therefore we have
$$\xi^r\text{-}\ind_{n-1}\Delta_{n-1}(\lambda-\epsilon) \rightarrow
\Delta_n(\mu)/N.$$ By adjointness of $\xi^r\text{-}\ind_{n-1}$ and
$\xi^r\text{-}\res_n$ we have
\begin{align*}
\Hom_{B^m_n}(\xi^r\text{-}\ind_{n-1}\Delta_{n-1}(\lambda-\epsilon),
\Delta_n(\mu)/N) 
\cong
\Hom_{B^m_{n-1}}(\Delta_{n-1}(\lambda-\epsilon),
\xi^r\text{-}\res_{n}\Delta_n(\mu)/N)  
\end{align*} 
By Proposition \ref{indres} we can conclude that either: 
\begin{align*}
[\Delta_{n-1}(\mu-\epsilon'):L_{n-1}(\lambda-\epsilon)]\neq 0
\end{align*} and $\epsilon' \in \xi^{r}\text{-}\remb(\mu)$ or
\begin{align*}
[\Delta_{n-1}(\mu+\epsilon''):L_{n-1}(\lambda-\epsilon)]\neq 0.
\end{align*} 
and $\epsilon'' \in \xi^{m-r}\text{-}\add(\mu)$.

In the first case we have by our inductive assumption
that $$|\mu-\epsilon'| \preceq_{(\sum_ia_i\xi^i)} |\lambda-\epsilon|$$
for some $a_i \geq0$.  We have that $\epsilon' \in
\xi^{r}\text{-}\remb(\mu)$ and $\epsilon \in
\xi^{r}\text{-}\remb(\mu)$ and so $$|\mu|\preceq_{(\sum_ia_i\xi^i)}
|\lambda|$$ as required. In the second case we have by our inductive
assumption that
$$|\mu + \epsilon''| \preceq_{(\sum_ia_i\xi^i)} |\lambda - \epsilon|$$
for some $a_i \geq0$.  We have that $\epsilon'' \in
\xi^{m-r}\text{-}\add(\mu)$ and $\epsilon \in
\xi^{r}\text{-}\remb(\mu)$ and so
$$|\mu | \preceq_{(\sum_ib_i\xi^i)} |\lambda |$$
where $a_i=b_i$ for all $i \neq r$, and $a_r=b_r-1$.
 \end{proof}

\begin{defn}
We define a partial order, $\leq$, on $\Lambda(m,n)$ by taking
$\lambda \leq \mu$ if $\lambda^i \subseteq \mu^i$ for all $0\leq i\leq m-1$
and $|\lambda|\preceq |\mu|$.
\end{defn}

\subsection{}\textbf{The restriction of standard modules to $H^m_n$}

Here we calculate the multiplicities
$$[\Delta_n(\lambda)\!  \downarrow_{H_n^m}:\mathbf{S}(\mu)]$$ 
for $\lambda\vdash n-2l$ and $\mu\vdash n$.  The case $l=1$ was already
done in \cite[Theorem 2.9]{ruixu}, where they remark (see \cite[Remark
  2.6]{ruixu}) that the general case given in \cite[Section
  4.4]{ruiyu} is incorrect.
 
\medskip

Recall that we have $\Delta_n(\lambda)=V(m,n,l)\otimes
\mathbf{S}(\lambda)$. From the explicit action of $B_n^m$ given in
Section 3.1, it is easy to see that we 
have
\begin{eqnarray*}
\Delta_n(\lambda)\! \downarrow_{H_n^m} &=& (V(m,n,l)\otimes
\mathbf{S}(\lambda))\! \downarrow _{H_n^m}\\ 
&\cong& H_n^m\otimes_{H_{2l}^m\otimes H_{n-2l}^m}(V(m,2l,l)\otimes
\mathbf{S}(\lambda)). 
\end{eqnarray*}
So the first step is to understand the structure of
$V(m,2l,l)\!\downarrow_{H_{2l}^m}$.

Each $(m,2l,l)$-dangle has $l$ arcs denoted by $(i_p,j_p)$ (for
$p=1,\ldots , l$) where $i_p$ (resp. $j_p$) is the left (resp. right)
vertex of the arc. Note that for any arc $(i_p,j_p)$ in $v$ and any
$r\in \mathbb{Z}/m\mathbb{Z}$ we have
\begin{equation}\label{pairing}
T_{i_p}^r v = T_{j_p}^{m-r} v.
\end{equation}
It follows that as a
$(\mathbb{Z}/m\mathbb{Z})^{2l}$-module, $V(m,2l,l)$ decomposes as
$$V(m,2l,l)\!\downarrow_{(\mathbb{Z}/m\mathbb{Z})^{2l}} =
\bigoplus_{\begin{subarray}{c}\tiny{v}\\ \tiny{(r_1,\ldots ,
      r_l)}\end{subarray}} k(T_{i_1}^{r_1}T_{i_2}^{r_2} \ldots
T_{i_l}^{r_l} v)$$ where the sum is over all $(m,2l,l)$-dangles $v$
with all arcs labelled by $0$ and over all $l$-tuples $(r_1,r_2,
\ldots, r_l)\in (\mathbb{Z}/m\mathbb{Z})^l$.  The generators
of $\Sigma_{2l}$ acts as follows: For each $p\neq q\in \{1,2,\ldots
l\}$ we have
\begin{eqnarray}
t_{i_p,j_p}T_{i_1}^{r_1}T_{i_2}^{r_2} \ldots T_{i_p}^{r_p} \ldots
T_{i_l}^{r_l} v &=& T_{i_1}^{r_1}T_{i_2}^{r_2} \ldots T_{i_p}^{m-r_p}
T_{i_l}^{r_l} v,\label{actionsym1}\\  
t_{i_p,i_q}T_{i_1}^{r_1}T_{i_2}^{r_2} \ldots T_{i_p}^{r_p} \ldots
T_{i_q}^{r_q} \ldots T_{i_l}^{r_l} v &=& T_{i_1}^{r_1}T_{i_2}^{r_2}
\ldots T_{i_p}^{r_q} \ldots T_{i_q}^{r_p} \ldots T_{i_l}^{r_l}
(t_{i_p,i_q}v).\label{actionsym2} 
\end{eqnarray}
For $(r_1,r_2, \ldots, r_l)\in (\mathbb{Z}/m\mathbb{Z})^l$
define the weight $wt(r_1,r_2,\ldots, r_l)$ to be 
$\varphi=(\varphi_0, \varphi_1, \ldots, \varphi_{\lfloor m/2\rfloor})$
where $$\varphi_i=|\{r_p\, : \, r_p=i \, \mbox{or}\, m-i\}|.$$ It
follows from (\ref{actionsym1}) and (\ref{actionsym2}) that
\begin{equation}\label{phidecomposition}
V(m,2l,l)\!\downarrow_{H_{2l}^m} = \oplus_\varphi V(m,2l,l)^\varphi
\end{equation}
where
$$V(m,2l,l)^\varphi =
\bigoplus_{\begin{subarray}{c}\tiny{v}\\ \tiny{wt(r_1,\ldots,
      r_l)=\varphi} \end{subarray}} k(T_{i_1}^{r_1}T_{i_2}^{r_2}
\ldots T_{i_p}^{r_p} \ldots T_{i_l}^{r_l} v)$$ and the sum is over all
$(m,2l,l)$-dangles $v$ with all arcs labelled by $0$.  Now, it follows
again from (\ref{actionsym1}) and (\ref{actionsym2}) that
$V(m,2l,l)^\varphi$ is a cyclic $H_{2l}^m$-module. We now construct an
explicit generator for this module. Let $v_\varphi$ be the
$(m,2l,l)$-dangle with all arcs labelled by $0$ and with set of arcs
given by
$$\cup_{0 \leq i \leq \lfloor{m/2} \rfloor}{\rm Arc}(i)$$
where we have
$${\rm Arc}(0)=\{(1,2),(3,4), \ldots , (2\varphi_0-1, 2\varphi_0)\}$$
and for $0<i<m/2$
\begin{eqnarray*}
{\rm Arc}(i)&=&\{(2(\varphi_0+\ldots +\varphi_{i-1})+1,2(\varphi_0
+\ldots \varphi_i)),\\
&&\quad (2(\varphi_0+\ldots +\varphi_{i-1})+2,2(\varphi_0
+\ldots \varphi_i)-1), \ldots \\ 
&&\quad\quad (2(\varphi_0+\ldots +\varphi_{i-1})+\phi_i,2(\varphi_0 +\ldots
\varphi_{i-1})+\varphi_i+1)\}, 
\end{eqnarray*}
and if $i=m/2$ we have
$$\{(2(\varphi_0+\ldots + \varphi_{m/2-1})+1, 2(\varphi_0+\ldots +
\varphi_{m/2-1})+2), \ldots , (2l-1,2l)\}.$$ 
The $(m,2l,l)$-dangle $v_{\varphi}$ is depicted in Figure \ref{figurevphi}.

\begin{figure}[ht]

\begin{center}
  \begin{minipage}{54mm}
\def\objectstyle{\scriptstyle}
\xymatrix@=2pt{
			&&		&&	&	&&	\ar@{..}[dd]	
			&&&&\ar@{..}[dd]	&&&&	\ar@{..}[dd]	&&	\ar@{..}[dd]	&&&&\ar@{..}[dd]&&&&\ar@{..}[dd]			\\
\circ 	\ar@{-}[rr]	& &\circ 		&\cdots& \circ 	\ar@{-}[rr]			&&\circ		&&\circ \ar@{-}@/_.7pc/[rrrrrr]		&\cdots&\circ \ar@{-}@/_.4pc/[rr]		&&\circ 		&\cdots&\circ		&&\cdots&& 
		 	 \circ \ar@{-}@/_.7pc/[rrrrrr]		&\cdots&\circ \ar@{-}@/_.4pc/[rr]		&&\circ 		&\cdots&\circ	
&&		\circ 	\ar@{-}[rr]	& &\circ 		&\cdots& \circ 	\ar@{-}[rr]			&&\circ			\\
&& &&&&&&&&&&&&&&&&&&&&&&&&&&&&	
} \end{minipage} 
 \end{center}
\caption{The $(m,2l,l)$-dangle $v_{\varphi}$.}
\label{figurevphi}
\end{figure}
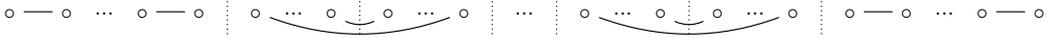

Now we define 
$$T^\varphi = \prod_{0\leq i\leq \lfloor m/2 \rfloor}
\prod_{(i_p,j_p)\in {\rm Arc}(i)}T_{i_p}^i.$$ It follows from
(\ref{actionsym1}) and (\ref{actionsym2}) that $T^\varphi v_\varphi$
is a generator for the $H_{2l}^m$-module $V(m,2l,l)$.  The stabiliser
of $k(T^\varphi v_\varphi)$ is given by
$${\rm Stab}(T^\varphi v_\varphi)=k((\mathbb{Z}/m\mathbb{Z})\wr
(\Sigma_2 \wr \Sigma_{\varphi_0}))\otimes \left(\bigotimes_{0<i<m/2}
k((\ZZ/m\ZZ)\wr \Sigma_{\varphi_i}) \right)\otimes k((\ZZ/m\ZZ)\wr
(\Sigma_2\wr \Sigma_{\varphi_{m/2}}))$$ where we ignore the last term
if $m$ is odd, and the group $\Sigma_{\varphi_i}$ is viewed as a
subgroup of $\Sigma_{\varphi_i}\times \Sigma_{\varphi_i}$ via the
diagonal embedding. As a module for its stabiliser, we have
$$k(T^\varphi v_\varphi)=(k_{\Sigma_2\wr
  \Sigma_{\varphi_0}})^{(0)}\otimes \left(\bigotimes_{0<i<m/2}((k\otimes
k)_{\Sigma_{\varphi_i}})^{(i)\otimes (m-i)}\right) \otimes (k_{\Sigma_2\wr
  \Sigma_{\varphi_{m/2}}})^{(m/2)}.$$ Thus we have
\begin{equation}\label{vphi}
V(m,2l,l)^\varphi \cong [(k_{\Sigma_2\wr
    \Sigma_{\varphi_0}})^{(0)}\otimes \left(\bigotimes_{0<i<m/2}((k\otimes
  k)_{\Sigma_{\varphi_i}})^{(i)\otimes (m-i)}\right) \otimes (k_{\Sigma_2\wr
    \Sigma_{\varphi_{m/2}}})^{(m/2)})]\! \uparrow_{{\rm
    Stab}(T^\varphi v_\varphi)}^{H_{2l}^m}. 
\end{equation}

\medskip

We can now prove the main result of this section.

\begin{thm}\label{lr}
Let $\lambda, \mu \in \Lambda(m,n)$.  If $\lambda \not \leq \mu$, then
$[\Delta_n(\lambda)\! \downarrow_{H_n^m}:\mathbf{S}(\mu)] =0$.
Otherwise, we have that 
$|\lambda| \preceq_{\sum a_r\xi^r} |\mu|$, and
\begin{align*}
[\Delta_n(\lambda)\!\downarrow_{H_n^m}:\mathbf{S}(\mu)] = \prod_{i \neq 0,
  m/2}\left(\sum_{\tau \vdash
  a_i}c^{\mu^i}_{\lambda^i,\tau}c^{\mu^{m-i}}_{\lambda^{m-i},\tau}\right)
\prod_{j = 0, m/2} \left(\sum _{\begin{subarray}{c}\eta\vdash 2a_j
    \\ \eta \,\,\mbox{\tiny{even}}\end{subarray}}
c^{\mu^j}_{\lambda^j, \eta}\right)
\end{align*}
\end{thm}

\begin{proof}
Using the decomposition of $V(m,2l,l)$ given in
(\ref{phidecomposition}) and (\ref{vphi}) and the construction of the
Specht modules given in Section 2.2 we have
\begin{eqnarray*}
\Delta_n(\lambda)\!\downarrow_{H_n^m}&\cong& (V(m,2l,l)\otimes
\mathbf{S}(\lambda))\!\uparrow_{H_{2l}^m\otimes H_n^m}^{H_n^m}\\ 
&\cong& \oplus_\varphi (V(m,2l,l)^\varphi\otimes
\mathbf{S}(\lambda))\!\uparrow_{H_{2l}^m\otimes H_n^m}^{H_n^m}\\ 
&\cong& \oplus_\varphi [ (k_{\Sigma_2\wr
    \Sigma_{\varphi_0}})^{(0)}\otimes \left(\bigotimes_{0<i<m/2}((k\otimes
  k)_{\Sigma_{\varphi_i}})^{(i)\otimes (m-i)}\right) \otimes (k_{\Sigma_2\wr
    \Sigma_{\varphi_{m/2}}})^{(m/2)}\\ 
&& \otimes \bigotimes_{1\leq i\leq m}S(\lambda^i)^{(i)}] \!
\uparrow_{{\rm Stab}(T^\varphi v_\varphi)\otimes
  H_{|\lambda|}^m}^{H_{n}^m}. 
\end{eqnarray*}
Rearranging according to the action of $(\ZZ/m\ZZ)^{n}$ and using
transitivity of induction we get
\begin{eqnarray*}
\Delta_n(\lambda)\!\downarrow_{H_n^m} &\cong& 
\oplus_\varphi [((k\!\uparrow_{\Sigma_2\wr
    \Sigma_{\varphi_0}}^{\Sigma_{2\varphi_0}}\otimes
  S(\lambda^m))\!\uparrow_{\Sigma_{2\varphi_0}\times
    \Sigma_{|\lambda^m|}}^{\Sigma_{2\varphi_0+|\lambda^m|}})^{(0)} \\ 
&& \otimes \left(\bigotimes_{0<i<m/2}(((k\otimes k)\!
  \uparrow_{\Sigma_{\varphi_i}}^{\Sigma_{\varphi_i}\times
    \Sigma_{\varphi_i}} \otimes S(\lambda^i)\otimes
  S(\lambda^{m-i})\!\uparrow_{\Sigma_{\varphi_i}\times
    \Sigma_{\varphi_i}\times \Sigma_{|\lambda^i|}\times
    \Sigma_{|\lambda^{m-i}|}}^{\Sigma_{\varphi_i+|\lambda^i|}\times
    \Sigma_{\varphi_i + |\lambda^{m-i}|}})^{(i)\otimes (m-i)}\right)\\ 
&& \otimes
  ((k\!\uparrow_{\Sigma_2\wr\Sigma_{\varphi_{m/2}}}^{\Sigma_{2\varphi_{m/2}}}\otimes
  S(\lambda^{m/2}))\!\uparrow_{\Sigma_{2\varphi_{m/2}}\times
    \Sigma_{|\lambda^{m/2}|}}^{\Sigma_{2\varphi_{m/2}+|\lambda^{m/2}|}})^{(m/2)}] 
\!\uparrow_{H_{\varphi+|\lambda|}^m}^{H_n^m},
\end{eqnarray*} 
where 
$$H_{\varphi + |\lambda|}^m=H_{2\varphi_0+|\lambda^m|}^m\otimes
\left(\bigotimes_{0<i<m/2} (H_{\varphi_i +|\lambda^i|}^m \otimes
H_{\varphi_i+|\lambda^{m-i}|}^m)\right) \otimes
H_{2\varphi_{m/2}+|\lambda^{m/2}|}^m.$$ 
Note that 
$$(k\!\uparrow_{\Sigma_2\wr
  \Sigma_{\varphi_0}}^{\Sigma_{2\varphi_0}}\otimes
S(\lambda^m))\!\uparrow_{\Sigma_{2\varphi_0}\times
  \Sigma_{|\lambda^m|}}^{\Sigma_{2\varphi_0+|\lambda^m|}}$$ is exactly
the restriction to $\Sigma_{2\varphi_0+|\lambda^m|}$ of the standard
module labelled by $\lambda^m$ for the classical Brauer algebra
$B(2\varphi_0+|\lambda^m|, \delta')$ (any parameter $\delta'$). And
similarly for the last term.  Note also that
$$((k\otimes k)\!
\uparrow_{\Sigma_{\varphi_i}}^{\Sigma_{\varphi_i}\times
  \Sigma_{\varphi_i}} \otimes S(\lambda^i)\otimes
S(\lambda^{m-i}))\!\uparrow_{\Sigma_{\varphi_i}\times
  \Sigma_{\varphi_i}\times \Sigma_{|\lambda^i|}\times
  \Sigma_{|\lambda^{m-i}|}}^{\Sigma_{\varphi_i+|\lambda^i|}\times
  \Sigma_{\varphi_i + |\lambda^{m-i}|}}$$ is exactly the restriction
to $\Sigma_{\varphi_i + |\lambda^i|}\times
\Sigma_{\varphi_i+|\lambda^{m-i}|}$ of the standard module labelled by
$(\lambda^i, \lambda^{m-i})$ for the walled Brauer algebra
$WB(\varphi_i + |\lambda^i|,\varphi_i+|\lambda^{m-i}|, \delta')$ (any
parameter $\delta'$). The result now follows from \cite{dhw} and
\cite{halvwall} (by replacing $\varphi_i$ by $a_i$ in the statement).
\end{proof}

\begin{cor}\label{neccondition}
Let $\lambda, \mu\in \Lambda(m,n)$. If $[\Delta_n(\lambda):
  L_n(\mu)]\neq 0$ then $\lambda \leq \mu$.
\end{cor}

\section{Truncation to idempotent subalgebras}
 
In this section we show that maximal co-saturated idempotent
subalgebras of $B^m_n$ are isomorphic to a tensor product of classical
and walled Brauer algebras. Hence we determine the space of
homomorphisms between standard modules and the decomposition numbers
for $B_n^m$.

\subsection{}\textbf{Co-saturated sets}
 
For $\omega\in \Lambda |m,n|$ an $m$-composition of $n$, we define
$(\preceq \omega)\subseteq \Lambda |m,n|$ to be the subset of all
$m$-compositions less than or equal to $\omega$ with respect to
$\preceq$. We define $\Lambda_\omega$ to be the pre-image of $(\preceq
\omega)$ in $\Lambda (m,n)$, that is the set of all $m$-partitions
$\lambda$ with $|\lambda|\preceq \omega$.

\begin{eg}
The diagram in Figure \ref{poset2} is the Hasse diagram of the poset
$(\preceq (1,1,4)) \subset \Lambda|3,6|$.  We have again annotated the
edges with the signed partial ordering.
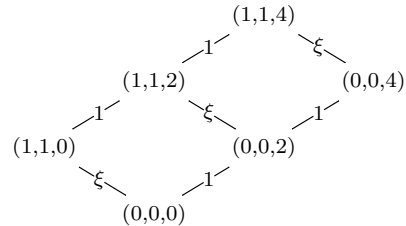
\begin{figure}[ht]
\centerline{
 \begin{minipage}{54mm}
\def\objectstyle{\scriptstyle}
\xymatrix@=2pt{
 		&&		&& 		&&
  &&(1,1,4)	\ar@{-}[ddrr]|\xi  \ar@{-}[ddll]|1 	&&
  &&	\\&\\ 
			&& &&				&&(1,1,2)
  \ar@{-}[ddrr]|\xi  \ar@{-}[ddll]|1
  &&&&(0,0,4)\ar@{-}[ddll]|1  \\&\\ 
			&&&&(1,1,0) && &&(0,0,2)  \\&\\
			&&&&&&(0,0,0)\ar@{-}[uurr]|1  \ar@{-}[uull]|\xi
}\end{minipage}}
\caption{A sub-poset of the Hasse poset in Figure \ref{poset1}.}
\label{poset2}
\end{figure}

\end{eg}

We have chosen to work with the partial order $\leq$ on $\Lambda(m,n)$
as it is a refinement of the natural partial ordering given by
inclusion on the set of multipartitions. Note, however, that
$B_n^m(\delta)$ is quasi-hereditary with respect to the opposite
partial order $\leq_{{\rm opp}}$ on $\Lambda(m,n)$, and we have that
$\Lambda_\omega\subseteq(\Lambda(m,n),\leq_{{\rm opp}})$ is a
co-saturated subset. So we can apply the results from
\cite[Appendix]{don2} on idempotent subalgebras corresponding to
co-saturated subsets for quasi-hereditary algebras.  The first thing
we need is an idempotent corresponding to $\Lambda_\omega$.

\subsection{}\textbf{Idempotents and standard modules}
 
In this section we consider the effect of applying the idempotents
$\pi_{{\omega}}$ defined in Section 2.2 to standard modules. 

Recall that an $(m,n,l)$-dangle $v$ can be described as a set of $l$
disjoint pairs $(i_p<j_p)\in \{1,\ldots , n\}^2$, called arcs, where
each arc is labelled by an element of $\ZZ/m\ZZ$. We say that $v$
belongs to $\omega$ if every arc $(i_p, j_p)$ (for $p=1,\ldots ,l$)
satisfies $i_p\in [\omega_{r_p}]$ and $j_p\in [\omega_{m-r_p}]$ for
some $0\leq r_p\leq \lfloor m/2 \rfloor$. In this case we define
$\omega \setminus v$ by
$$\omega \setminus v = \omega - \sum_{p=1}^l
(\epsilon_{r_p}+\epsilon_{m-r_p})\in \Lambda|m,n|$$ where
$\epsilon_{r_p}=(0,\ldots , 0 ,1, 0 ,\ldots , 0)$ with a $1$ in
position $r_p$ (and similarly for $\epsilon_{m-r_p}$).

\begin{prop}\label{omegastandard}
Let $v$ be an $(m,n,l)$-tangle, $\lambda\in \Lambda(m,n)$ and $x\in
\mathbf{S(\lambda)}$. Then we have
$$ \pi_\omega (v\otimes x) = \left\{ \begin{array}{ll}
  T_{i_1}^{r_1}t_{i_2}^{r_2}\ldots T_{i_l}^{r_l}v\otimes \pi_{\omega
    \setminus v}x & \mbox{if $v$ belongs to $\omega$}\\ 0 &
  \mbox{otherwise.}\end{array}\right.$$ 
In particular, we have that
$\pi_\omega \Delta_n(\lambda)\neq 0$ if and only if $\lambda \in
\Lambda_\omega$.
\end{prop}

\begin{proof}
Note that for each arc ($i_p, j_p)$ of $v$ we have
$$T_{i_p}^{r}v=T_{j_p}^{m-r}v.$$ Now as $\{T_i^r\, : \, 0\leq r\leq
m-1\}$ form a set of orthogonal idempotents we have
$$T_{i_p}^rT_{j_p}^sv=T_{i_p}^{r}T_{i_p}^{m-s}v=\left\{ \begin{array}{ll}
  T_{i_p}^rv & \mbox{if $s=m-r$}\\ 0 &
  \mbox{otherwise.}\end{array}\right.$$ Thus we have that $\pi_\omega
(v\otimes x)=0$ unless $v$ belongs to $\omega$. Now it is easy to see
that $\pi_\omega$ acts on the free vertices of $v$ by $\pi_{\omega
  \setminus v}$. So if $v$ belongs to $\omega$ we get
$$\pi_\omega(v\otimes x) = T_{i_1}^{r_1}t_{i_2}^{r_2}\ldots
T_{i_l}^{r_l}v\otimes \pi_{\omega \setminus v}x $$ 
as required.
Now using Section 2.2 we have that 
$$\pi_{\omega \setminus v}\mathbf{S(\lambda)}\cong
\left\{ \begin{array}{ll} S(\lambda^0)\otimes S(\lambda^1)\otimes
  \ldots \otimes S(\lambda^{m-1}) & \mbox{if $|\lambda|=\omega
    \setminus v$} \\ 0 & \mbox{otherwise.} \end{array} \right.$$ 
Finally note that $\lambda\in \Lambda_\omega$ if and only if
$|\lambda|=\omega \setminus v$ for some $v$ belonging to
$\omega$. This proves the last part of the proposition. 
\end{proof}

\subsection{}\textbf{Truncation functors} 

We now consider the truncation functor defined by the idempotent $\pi_\omega$.
From now on we shall denote $ \pi_{\omega}
B^m_n \pi_{\omega}$ by $B^m_{\omega}$.  The truncation functor is defined by
\begin{align*}
f_{\omega} 	: &B^m_n{\text{\rm{-mod}}} \to B^m_{
  \omega}{\text{\rm{-mod}}} \\ 
				& M \longmapsto \pi_{\omega}M.
\end{align*}

Using Proposition \ref{omegastandard} and \cite[A3.11]{don2} we have
the following result.

\begin{prop}\label{qh}
(i) A complete set of non-isomorphic simple $B_\omega^m$-modules is
  given by $$\{f_\omega L_n(\lambda) \, : \, \lambda \in
  \Lambda_\omega\}.$$ 
(ii) A complete set of non-isomorphic indecomposable projective
  $B_\omega^m$-modules is given by $$\{f_\omega P_n(\lambda) \, : \,
  \lambda \in \Lambda_\omega\}.$$ 
(iii) The algebra $B^m_{\omega}$ is a
quasi-hereditary algebra with respect to the partial order $\leq_{{\rm
    opp}}$ on  
$\Lambda_{{\omega}}$.  Its standard modules are given by
$f_\omega\Delta_n(\lambda)$
for all $\lambda\in \Lambda_\omega$. 
\end{prop}

\noindent For $M\in B_n^m{\text{\rm{-mod}}}$ we write $M\in
\mathcal{F}_\omega(\Delta)$ to indicate that $M$ has a filtration with
subquotients belonging to $\{\Delta_n(\lambda)\, : \,
\lambda\in\Lambda_\omega\}$.  

\begin{prop}[A3.13 \cite{don2}]\label{exty}
Let $X, Y\in B_n^m{\text{\rm{-mod}}}$ with  $X\in
\mathcal{F}_\omega(\Delta)$. For all $i\geq 0$ we have  
$$\Ext^i_{B^m_n}(X,Y)\cong \Ext^i_{B^m_\omega}(f_\omega X,f_\omega Y).$$ 
\end{prop}

\noindent As $P_n(\lambda)\in \mathcal{F}_\omega(\Delta)$ for all
$\lambda\in \Lambda_\omega$ and $[\Delta_n(\mu):L_n(\lambda)]=\dim
\Hom (P_n(\lambda), \Delta_n(\mu))$ we have the following corollary. 

\begin{cor}\label{truncdec}
For all $\lambda, \mu\in \Lambda_\omega$ we have
$$[\Delta_n(\mu):L_n(\lambda)]=[f_\omega \Delta_n(\mu): f_\omega
  L_n(\lambda)].$$ 
\end{cor} 

\subsection{}\textbf{The idempotent subalgebras $B^m_{{{\omega}}}$} 

We now wish to understand the structure of these idempotent
subalgebras.
Therefore we start by considering the image, in $B_{{\omega}}^m$, of the
generators of the cyclotomic Brauer algebra $B_n^m$.

\begin{lem}\label{imagegen}
Let $1\leq i,j \leq n$ and let $\omega$ be an $m$-composition of
$n$. Then we have\\ 
(i) $\pi_\omega t_i^k\pi_\omega = \xi^{-kr}\pi_\omega$ if $i\in
[\omega_r]$ for some $0\leq r\leq m-1$.\\ 
(ii) $\pi_\omega t_{i,j}\pi_\omega\neq 0$ if and only if $i,j\in
[\omega_r]$ for some $0\leq r\leq m-1$.\\ 
(iii) $\pi_\omega e_{i,j}\pi_\omega\neq 0$ if and only if $i\in
[\omega_r]$ and $j\in [\omega_{m-r}]$ for some $0\leq r\leq m-1$. 
\end{lem}

\begin{proof}
This follows from the definition of $\pi_\omega$, equation
(\ref{pairing}) and the fact that the $T_i^r$'s for $0\leq r\leq m-1$
(and fixed $i$) are orthogonal idempotents.
\end{proof}

We now state the main result of this section.

\begin{thm}\label{isom} Let $\omega$ be an $m$-composition of $n$.
The algebra $B^m_{{{\omega}}}$ is isomorphic to a product of Brauer
and walled Brauer algebras with parameters $\overline{\delta}_r$ for
$0\leq r\leq \lfloor m/2\rfloor$.  More specifically
$$B^m_{{{\omega}}} \cong
B({\omega_{0}},\overline{\delta}_0)
\otimes \bigotimes_{r=1}^{\lfloor m/2\rfloor} WB(\omega_{r},
\omega_{m-r},\overline{\delta}_r) 
$$
if $m$ is odd, and
$$B^m_{{{\omega}}} \cong
B({\omega_{0}},\overline{\delta}_0)
\otimes \left( \bigotimes_{r=1}^{  (m/2)-1} WB(\omega_{r},
\omega_{m-r},\overline{\delta}_r) 
\right) \otimes 
B({\omega_{m/2}},\overline{\delta}_{m/2})
$$ if $m$ is even.  
\end{thm}

\begin{rem}
In our definition of multiplication for $B^m_n$ 
we chose one of two possible orientations of the closed loops. 
Had we favoured the alternative orientation, the above proposition
would be stated in terms of the conjugate parameters
$\overline{\delta}_r$ such that  $m/2 \leq
r \leq m-1$. 
This makes no difference to the representation theory as we obtain
non-semisimple specialisations only when these parameters are integral
--- in which case $\overline{\delta}_r=\overline{\delta}_{m-r}$.
\end{rem}

\begin{proof} We will assume that $m$ is even in the proof. The case
  $m$ odd is obtained by ignoring all the terms corresponding to
  $m/2$. 

We view the tensor product of Brauer and walled Brauer algebras as a
diagram algebra spanned by certain Brauer diagrams with $n$ northern
and southern nodes. More precisely, as vector spaces, we embed
$B({\omega_{0}},\overline{\delta}_0) \otimes \left(
\bigotimes_{0<r<m/2} WB(\omega_{r}, \omega_{m-r},\overline{\delta}_r)
\right) \otimes B({\omega_{m/2}},\overline{\delta}_{m/2})$ into the
vector space $B(n)$ by partitioning the $n$ northern and southern
nodes according to $\omega$, that is we draw a wall after the first
$\omega_0$ nodes, then another wall after the next $\omega_1$ nodes,
etc. We embed the diagrams in $B(\omega_0, \overline{\delta_0})$ using
the first $\omega_0$ northern and southern nodes. For $0<r<m/2$ we
embed the diagrams in $B(\omega_r, \omega_{m-r},\overline{\delta_r})$
using all nodes $i, \bar{i}\in [\omega_r]$ or
$[\omega_{m-r}]$. Finally, we embed $B(\omega_{m/2},
\overline{\delta}_{m/2})$ using all nodes $i, \bar{i}\in
         [\omega_{m/2}]$. An example of such a digram is given in
         Figure \ref{embedding}.

\begin{figure}[ht]
\includegraphics[width=8cm]{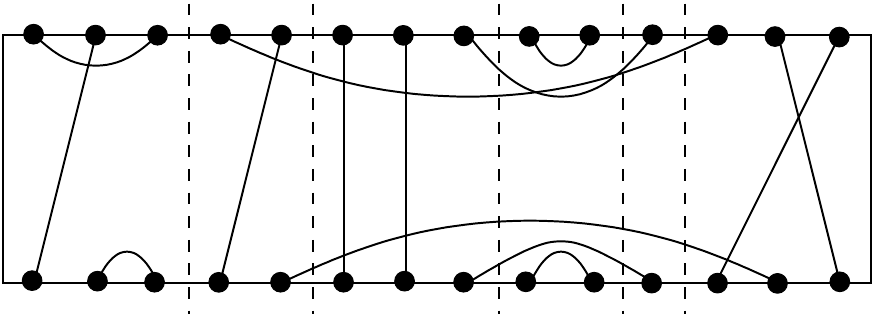}
\caption{An example of the embedding of $$B(3,
  \overline{\delta}_0)\otimes WB(2,3,\overline{\delta}_1) \otimes
  WB(3,1,\overline{\delta}_2) \otimes B(2, \overline{\delta}_3)$$ into
  $B(14)$ corresponding to the $6$-composition $\omega$ of $14$ given
  by $\omega=(3,2,3,2,1,3)$.}
\label{embedding}
\end{figure}

Now the multiplication is given by concatenation. Note that each
closed loop obtained by concatenation only contains nodes $i\in
[\omega_r]$ or $[\omega_{m-r}]$ for some $0\leq r\leq \lfloor
m/2\rfloor$; we then remove this closed loop and multiply by the
scalar $\overline{\delta_r}$.

We denote by $\sigma_{i,j}$, resp. $u_{i,j}$, the unoriented version
of $t_{i,j}$, resp. $e_{i,j}$. Now define the map
$$\phi:B^m_{{{\omega}}} \rightarrow
B({\omega_{0}},\overline{\delta}_0) \otimes \left(
\bigotimes_{0<r<m/2} WB(\omega_{r}, \omega_{m-r},\overline{\delta}_r)
\right) \otimes B({\omega_{m/2}},\overline{\delta}_{m/2})$$ on
generators by setting $\phi(\pi_\omega)=1$, $\phi(\pi_\omega
t_{i,j}\pi_\omega)=\sigma_{i,j}$ for all $i<j\in [\omega_r]$ for some
$0\leq r\leq m-1$, and $\phi(\pi_\omega e_{i,j}\pi_\omega)=u_{i,j}$
for all $i<j$ with $i\in [\omega_r]$ and $j\in [\omega_{m-r}]$ for
some $0\leq r\leq \lfloor m/2 \rfloor$. It is clear from the
description in terms of diagrams that $\phi$ gives a bijection and
that all the relations involving only the $\sigma_{i,j}$'s, or the
$\sigma_{i,j}$'s and the $u_{i,j}$'s are satisfied. It remains to show
that for $i\in [\omega_r]$ and $j\in [\omega_{m-r}]$ we have
$(\pi_\omega e_{i,j}\pi_\omega)^2=\overline{\delta_r} (\pi_\omega
e_{i,j}\pi_\omega)$. Now we have
\begin{eqnarray*}
(\pi_\omega e_{i,j}\pi_\omega)^2 &=& \pi_\omega e_{i,j}\pi_\omega e_{i,j} \pi_\omega\\
&=& \pi_\omega e_{i,j} T_i^r T_j^{m-r} e_{i,j} \pi_\omega\\
&=& \pi_\omega e_{i,j} (T_i^r)^2e_{i,j} \pi_\omega\\
&=& \pi_\omega e_{i,j} T_i^r e_{i,j}\pi_\omega\\
&=& \sum_{a=0}^{m-1} \xi^{ar} \pi_\omega e_{i,j} t_i^a e_{i,j}\pi_\omega 
= \sum_{a=0}^{m-1} \xi^{ar}  \delta_a \pi_\omega e_{i,j} \pi_\omega 
= \overline{\delta_r} (\pi_\omega e_{i,j} \pi_\omega).
\end{eqnarray*}
\end{proof}

\subsection{}\textbf{Homomorphisms and decomposition numbers} 

Recall that the standard modules for the classical Brauer algebra
$B(n,\delta)$ are indexed by partitions $\lambda$ of $n-2l$ for $0\leq
l \leq \lfloor n/2 \rfloor$. For each partition $\lambda$ of $n-2l$,
the standard $B(n, \delta)$-module $\Delta_{B(n)}(\lambda)$ can be
constructed by inflating the Specht module $S(\lambda)$ along
$V(1,n,l)$, and it has simple head $L_{B(n)}(\lambda)$. The standard
modules for the walled Brauer algebra $WB(r, s,\delta)$ are indexed by
bi-partitions $(\lambda, \mu)$ of $(r-l, s-l)$ for $0\leq l\leq {\rm
  min}\{r, s\}$.  For each bi-partition $(\lambda, \mu)$, the standard
$WB(r,s,\delta)$-module $\Delta_{WB(r,s)}(\lambda, \mu)$ can be
constructed similarly by inflating the tensor product of Specht
modules $S(\lambda)\otimes S(\mu)$ along the corresponding subspace of
dangles, and it has simple head $L_{WB(r,s)}(\lambda, \mu)$.

\begin{prop}\label{truncstandard}
For $\lambda\in \Lambda_\omega$ the module
$\Delta_n^\omega(\lambda)=f_\omega \Delta_n(\lambda)$ 
is isomorphic to
$$\Delta_{B(\omega_0)}(\lambda^0)\otimes \bigotimes_{1\leq r\leq
  \lfloor m/2\rfloor} \Delta_{WB(\omega_r, \omega_{m-r})}(\lambda^r,
\lambda^{m-r}) \otimes \Delta_{B(\omega_{m/2})}(\lambda^{m/2})$$
(under the isomorphism given in Theorem \ref{isom}) where we ignore
the last term when $m$ is odd.
\end{prop}

\begin{proof}
By Proposition \ref{qh}(iii), we know that $f_\omega
\Delta_n(\lambda)$ is a standard module. So we only need to show that
it is labelled by the same partition. Now the required tensor product
of standard modules is characterised by the fact that when we localise
this module to
$$B(|\lambda^0|)\otimes \left(\bigotimes_{1\leq r\leq \lfloor m/2
  \rfloor}WB(|\lambda^r|, |\lambda^{m-r}|)\right) \otimes B(|\lambda^{m/2}|)$$
we get a module isomorphic to 
$$S(\lambda^0)\otimes \left(\bigotimes_{1\leq
  r\leq \lfloor m/2 \rfloor} (S(\lambda^r)\otimes
S(\lambda^{m-r}))\right)\otimes S(\lambda^{m/2}).$$
 But it is clear that
$f_\omega \Delta_n(\lambda)$ satisfies this condition using
Proposition \ref{omegastandard} and Section \ref{omegaspecht}.
\end{proof}

\begin{cor}
Let $\lambda, \mu\in \Lambda(m,n)$ and define $\omega=|\lambda|$.  Then
(i) $\Hom_{B^m_n}(\Delta_n(\lambda), \Delta_n(\mu))$ is isomorphic
to 
\begin{eqnarray*}
\Hom_{B(\omega_0,\overline{\delta_0})}(\Delta_{B}(\lambda^0),
\Delta_{B}(\mu^0)) \otimes \bigotimes_{0<r<m/2}\Hom_{WB(\omega_r,
  \omega_{m-r},\overline{\delta_r})}(\Delta_{WB}(\lambda^r,
\lambda^{m-r}), \Delta_{WB}(\mu^r, \mu^{m-r})) && \\ \otimes
\Hom_{B(\omega_{m/2},\overline{\delta_{m/2}})}(\Delta_{B}(\lambda^{m/2}),
\Delta_{B}(\mu^{m/2})). && 
\end{eqnarray*}  
(ii) The decomposition numbers $[\Delta_n(\mu):L_n(\lambda)]$ factorise as
$$[\Delta_{B}(\mu^0): L_{B}(\lambda^0)] \times \prod_{0<r<m/2}
    [\Delta_{WB}(\mu^r, \mu^{m-r}): L_{WB}(\lambda^r, \lambda^{m-r})]
    \times [\Delta_{B}(\mu^{m/2}): L_{B}(\lambda^{m/2})].$$ 
(We ignore the last term in (i) and (ii) when $m$ is odd).
\end{cor}
 \begin{proof} By localisation (\ref{localisation}), we can always
   assume that $\lambda$ is an $m$-partition of $n$. Now,
we have seen in Corollary \ref{neccondition} that a necessary
condition for a non-zero homomorphism (or decomposition number) is
that $\lambda \geq \mu$. Thus we have $\lambda, \mu\in \Lambda_\omega$
where $\omega = |\lambda|$.  We then obtain the results using
Propositions \ref{truncstandard} and \ref{exty} and Corollary
\ref{truncdec}.
\end{proof}

\begin{rem} Let $\omega$ be an $m$-composition of $n$ and let
  $\lambda, \mu\in \Lambda_\omega$. Then, using Proposition
  \ref{exty}, we have more generally that
$$\Ext^i_{B_n^m}(\Delta_n(\lambda),\Delta_n(\mu))\cong
  \Ext^i_{B_{\omega}^m}(\Delta_n^\omega(\lambda),
  \Delta_n^\omega(\mu))$$ for all $i\geq 0$.
\end{rem}

\begin{rem}
(i) With this factorisation of the decomposition numbers at hand, one can
easily deduce the block structure of $B^m_n$ (in terms of that of the
walled and classical Brauer algebras).  

(ii) The decomposition numbers for the Brauer and
walled Brauer algebras in characteristic zero are known by
\cite{marbrauer} and \cite{cdv}, and so we have determined the
decomposition numbers for the cyclotomic Brauer algebra in
characteristic zero.  
\end{rem}

\section{Appendix: The unoriented cyclotomic Brauer algebra}

There is another version of the cyclotomic Brauer algebra, which we
will denote by $\tilde{B}_n^m(\delta)$, spanned by unoriented
reduced $(m,n)$-diagrams. As a vector space, it coincides with $B_n^m$
but the multiplication is simply given by concatenation, addition (in
$\ZZ/m\ZZ$) of the labels on each strands, and replacing any closed
loop labelled by $r$ with scalar multiplication by $\delta_r$.

All the arguments in this paper apply to the unoriented cyclotomic
Brauer algebra as well, and it turns out that the corresponding
idempotent subalgebras are isomorphic to a tensor product of classical
Brauer algebras in this case. Hence this gives a factorisation of the
decomposition numbers of $\tilde{B}_n^m$ as a product of decomposition
numbers for the classical Brauer algebras. We will now briefly sketch
the modifications required.

The algebra $\tilde{B}_n^m$ is still an iterated inflation of the
algebras $H_n^m$ but along the spaces of unoriented dangles. All the
results in Section 3 hold as before if we replace $m-r$ by $r$ in
Lemma \ref{A4}(ii) and Proposition \ref{indres}. In Section 4, note
that in equation (\ref{pairing}) and equation (\ref{actionsym1}) we
have to replace $m-r$ with $r$ again. Now following the argument in
Section 4.2 we obtain
$$[\Delta_n(\lambda)\!\downarrow _{H_n^m}:\mathbf{S}(\mu)] =
\prod_{0\leq j\leq m-1}\sum_{\begin{subarray}{c}\eta\vdash 2a_j
    \\ \eta \, \mbox{{\tiny
        even}}\end{subarray}}c_{\lambda^j,\eta}^{\mu^j}.$$ We modify
the partial ordering $\preceq$ and $\leq$ accordingly. For $\omega,
\omega'\in \Lambda|m,n|$ we set $\omega \preceq \omega'$ if and only
if $\omega_r-\omega'_r\geq 0$ and $\omega_r-\omega'_r$ is even for all
$0\leq r\leq m-1$. We then define $\leq$ on $\Lambda(m,n)$ by setting
$\lambda \leq \mu$ if and only if $|\lambda|\preceq |\mu|$ and
$\lambda^r\subseteq \mu^r$ for all $0\leq r\leq m-1$. We then have
that Corollory \ref{neccondition} holds with respect to this new
partial order. In Section 5 we can define the co-saturated subset
$\Lambda_\omega$ as before (using the new partial order). Replacing
$m-r$ with $r$ throughout the arguments we obtain that
$$\tilde{B}_\omega^m\cong \bigotimes_{0\leq r\leq m-1}B(\omega_r,
\overline{\delta_r}),$$ and hence we get the required factorisation of
homomorphisms between standard modules and of decomposition numbers.

\providecommand{\bysame}{\leavevmode\hbox to3em{\hrulefill}\thinspace}
\providecommand{\MR}{\relax\ifhmode\unskip\space\fi MR }
% \MRhref is called by the amsart/book/proc definition of \MR.
\providecommand{\MRhref}[2]{%
  \href{http://www.ams.org/mathscinet-getitem?mr=#1}{#2}
}
\providecommand{\href}[2]{#2}

\end{document}